\documentclass[preprint, 11pt]{amsart}

\usepackage{amssymb,amscd}
 \usepackage{amsthm}

\begin{document}
 





\title{Prescribed subintegral extensions of local  Noetherian domains}

\author{Bruce Olberding}

\address{Department of Mathematical Sciences, New Mexico State University,
Las Cruces, NM 88003-8001,
olberdin@nmsu.edu}





\maketitle


\begin{abstract} We show how subintegral extensions of certain local Noetherian domains $S$ can be constructed with specified 
 invariants including reduction number,  Hilbert function, multiplicity and local cohomology. The construction behaves analytically like Nagata idealization but rather than a ring extension of $S$, it produces a subring $R$ of $S$ such that $R \subseteq S$ is subintegral. \end{abstract}




\newtheorem{thm}{Theorem}[section]
\newtheorem{lem}[thm]{Lemma}
\newtheorem{prop}[thm]{Proposition}
\newtheorem{cor}[thm]{Corollary}
\newtheorem{exmp}[thm]{Example}
\newtheorem{rem}[thm]{Remark}
\newtheorem{ques}[thm]{Question}
\newtheorem{defn}[thm]{Definition}

\input amssym.def

\newcommand{\ilim}{\mathop{\varinjlim}\limits}

\def\Jac{\mbox{\rm Jac$\:$}}
\def\Nil{\mbox{\rm Nil$\:$}}
\def\wt{\widetilde}
\def\Q{\mathcal{Q}}
\def\O{\mathcal{O}}
\def\ff{\frak}
\def\Spec{\mbox{\rm Spec}}
\def\ZS{\mbox{\rm ZS}}
\def\wB{\mbox{\rm wB}}
\def\type{\mbox{ type}}
\def\Hom{\mbox{ Hom}}
\def\rank{\mbox{ rank}}
\def\Ext{\mbox{ Ext}}
\def\Ker{\mbox{ Ker }}
\def\Max{\mbox{\rm Max}}
\def\End{\mbox{\rm End}}
\def\ord{\mbox{\rm ord}}
\def\l{\langle\:}
\def\r{\:\rangle}
\def\Rad{\mbox{\rm Rad}}
\def\Zar{\mbox{\rm Zar}}
\def\Supp{\mbox{\rm Supp}}
\def\Rep{\mbox{\rm Rep}}
\def\cal{\mathcal}
\def\p{{\rm{p}}}

\def\gen{\mbox{\rm gen$\:$}}
\def\Jac{\mbox{\rm Jac$\:$}}
\def\Nil{\mbox{\rm Nil$\:$}}
\def\ord{\mbox{\rm ord}}
\def\wt{\widetilde}
\def\Q{\mathcal{Q}}
\def\ff{\frak}
\def\Spec{\mbox{\rm Spec}}
\def\ZS{\mbox{\rm ZS}}
\def\wB{\mbox{\rm wB}}
\def\type{\mbox{ type}}
\def\Hom{{\rm Hom}}
\def\depth{{\rm depth}}
\def\Gen{{\rm Gen}}
\def\rank{\mbox{\rm rank}}
\def\Ext{{\rm Ext}}
\def\embdim{{\mbox{\rm emb.dim } }}
\def\length{{\mbox{\rm length }}}
\def\Ker{\mbox{\rm Ker }}
\def\pd{{\rm pd}}
\def\gr{{\rm gr}}
\def\Max{\mbox{\rm Max}}
\def\End{\mbox{\rm End}}
\def\l{\langle\:}
\def\r{\:\rangle}
\def\Rad{\mbox{\rm Rad}}
\def\Zar{\mbox{\rm Zar}}
\def\Supp{\mbox{\rm Supp}}
\def\Rep{\mbox{\rm Rep}}
\def\cal{\mathcal}
\def\p{{\rm{p}}}

\section{Introduction}

In \cite{Swan}, \index{Swan, R.}
Swan introduced the  notion of a subintegral extension of rings to study the  $K$-theory of seminormal extensions.
 An extension of
(commutative) rings 
 $R \subseteq S$ is {\it subintegral} if it is
integral, the contraction mapping $\Spec(S) \rightarrow \Spec(R)$ is
a bijection and the induced maps on residue field extensions are
isomorphisms. In particular, when $R\subseteq S$ is a subintegral extension of domains, then $R$ and $S$ share the same quotient field. 
Alternatively, an integral extension $R \subseteq S$ is subintegral if and only if
for all homomorphisms $\phi:R \rightarrow K$ into a field $K$, there
exists a unique extension $\phi':S \rightarrow K$ \cite[Lemma 2.1]{Swan}.
In this article we work out some of the  local algebra for a special class of subintegral extensions of Noetherian rings developed in \cite{OlbCounter}, with the goal of  showing how given a local Noetherian ring $S$, a subintegral extension  $R \subseteq S$ can be found with prescribed invariants. The local rings $R$ we consider are interesting also in that they do not have finite normalization and hence are analytically ramified. Thus a secondary motivation is to construct analytically ramified local Noetherian domains that when taken as a pair with their normalization have prescribed  invariants. (A survey of local Noetherian domains without finite normalization can be found in \cite{OSurvey}.) 
In particular, 
we
consider  invariants  such as reduction number, Hilbert function, multiplicity, and local cohomology, for ideals $I$ in  subintegral extensions $R \subseteq S$ of Noetherian rings with the property that $I$ is {\it contracted} from $S$, meaning that  $I = IS \cap R$.   The class of contracted ideals of $R$ includes the prime ideals, as well as the integrally closed ideals of $R$.  Our motivation is to show how these invariants can be prescribed for certain subintegral extensions of Noetherian domains.

The  setting on which we focus has its origins in a construction of analytically ramified local Noetherian domains due to Ferrand and Raynaud \cite{FR}. Roughly, the idea is to construct from a ring $S$ and an $S$-module $K$ a subring $R$ of $S$ that reflects the structure of $S$ and $K$. This is done using a derivation: If  
  $S$ is a ring, $L$ is an $S$-module and $D:S \rightarrow L$ is a derivation, then for any $S$-submodule $K$ of $L$,  $R:=D^{-1}(K)$ is a subring of $S$. By carefully choosing $S$, $L$ and $D$, Ferrand and Raynaud exhibit interesting examples of analytically ramified local Noetherian domains of dimensions $1$ and $2$. A different version of their construction is given in \cite{OlbCounter} in order to produce analytically ramified local Noetherian domains in higher dimensions, and it is this construction that gives rise to the rings we consider here.  While the article \cite{OlbCounter} was concerned with elementary properties of the construction, in this article we consider in more depth the local algebra of these rings.

In the  specific situation on which we focus,  $R$ is  determined by the ring  $S$, an $S$-module $K$ and a derivation on $S$.      
Let $S$ be a ring, let $K$ be an $S$-module, and let $C$ be a multiplicatively closed subset of nonzerodivisors of $S$ that are also nonzerodivisors on $K$. 
  Then a subring $R$ of $S$ is {\it twisted by $K$ along $C$} if there is a  derivation $D:S_C \rightarrow K_C$
   such that $D(c) = 0$ for all $c \in C$ (i.e., $D$ is ``$C$-linear'') and the following properties hold: (a)  $R = S \cap D^{-1}(K)$,
(b)
 $D(S_C)$ generates $K_C$ as an $S_C$-module; and
(c)  $S \subseteq  \Ker D + cS$ for all $c \in C$.
 When it is necessary to specify the derivation $D$, we say that {\it $D$ twists $R$ by $K$ along $C$}. There is also an absolute version of this definition when $S$ is a domain with quotient field $F$, one for which no reference to $C$ is needed: Let $K$ be a torsion-free $S$-module, and let $FK$ denote the divisible hull $F \otimes_S K$ of $K$.  We say that
{\it $R$ is  strongly twisted by  $K$} if there is a derivation $D:F \rightarrow FK$ such that:
(a)  $R = S \cap D^{-1}(K)$;
(b)
 $D(F)$ generates $FK$ as an $F$-vector space;
and
(c)  $S \subseteq  \Ker D + sS$ for all $0\ne s \in S$.
 As above, we say that {\it $D$ strongly twists $R$ by $K$.}
It is not hard to see that strongly twisted implies twisted along $C = A \setminus \{0\}$.  

Twisted subrings were studied in \cite{OlbCounter,OlbAR}, and we recall some of their properties in Section 2. The main existence result  from \cite[Theorem 3.5]{OlbCounter} is

\medskip

{\bf (1.1)} {\it Existence of twisted subrings.}
Let $F/k$ be a  separably generated field extension of infinite transcendence degree such that $k$ has  characteristic $p \ne 0$ and at most countably many elements.     If $S$ is a $k$-subalgebra of $F$ with quotient field $F$ and $K$ is a torsion-free $S$ module of at most countable rank, then there exists a subring $R$ of $S$ that is strongly twisted by $K$.

\medskip

For example, suppose $L$ is a field of positive characteristic that is separably generated and of infinite transcendence degree over a countable subfield $k$. Let $X_1,\ldots,X_d$ be indeterminates for $L$. Then for  any ring $S$ between $L[X_1,\ldots,X_d]$ and $F:=L(X_1,\ldots,X_d)$, and torsion-free $S$-module $K$ of at most countable rank, there exists a subring $R$ of $S$ strongly twisted by $K$ (but $R$ need not be an $L$-subalgebra). Hence, as discussed in Section 2, the extension $R \subseteq S$ is subintegral, and it is such extensions on which we focus in this article.  The general strategy is to construct subintegral extensions $R \subseteq S$ by choosing $S$ in accordance with (1.1), as well as a torsion-free $S$-module $K$, then considering the subring $R$ of $S$ strongly twisted by $K$. The objects  $S$ and $K$ determine $R$, and the local data of $R$ is calculated in terms of invariants of $S$ and $K$.

As discussed in Section 2, when $S$ is a local Noetherian domain and $K$ is chosen carefully enough (e.g., $K$ is a finitely generated $S$-module), then $R$ is also a local Noetherian domain. Moreover, there is an isomorphism of rings $\widehat{R} \rightarrow \widehat{S} \star \widehat{K}$, where $\widehat{(-)}$ represents completion in the ${\bf m}$-adic topology and $\star$ is Nagata idealization (see Section 2).  
Some of the local algebra of idealizations has been studied, for example, by Valtonen in \cite{Val}. 
The goal of this article is to conduct a similar investigation for twisted subrings, which can be viewed as a kind of inversion of idealization. Since the completion of a twisted subring is an idealization,  some of the local algebra of idealizations allows one to deduce quickly some of the local algebra of twisted subrings, such as embedding dimension and multiplicity of the ring. But the multiplicity of non-maximal ideals of the ring is more subtle. The reason that non-maximal ideals  cannot be dealt with similarly is that  $\widehat{R}$ is isomorphic {as a ring} to $\widehat{S} \star \widehat{K}$, but not as an $R$-algebra, at least not with the standard $R$-module structure on $\widehat{S} \star \widehat{K}$ lifted from the $R$-module $S \oplus K$. (Lemma~\ref{product lemma} illustrates one aspect of this complication.) 
The isomorphism in fact has a ``twist'' to it, due to the presence of a derivation. This twist is mainly what this article   accounts for. 

\medskip

{\it A note on generality.} 
The existence theorem (1.1) is our main source of examples for (strongly) twisted subrings in dimension $>1$ (the articles \cite{OlbAR, OlbNag}  deal directly with the more tractable one-dimensional case). 
 Moreover, there is a straightforward characterization in (2.9) of when a strongly twisted subring is Noetherian. This is in contrast to  the more general case of being twisted along a multiplicatively closed set $C$, where only for ideals meeting $C$ can  finite generation be detected; see \cite[Lemma 5.1]{OlbCounter}. 
   However, as noted in Example~\ref{UFD example}, there is at least one case in which {\it Noetherian} subrings twisted along a multiplicatively closed set, rather than strongly twisted, can be produced. For this reason, rather than work with the notationally simpler strongly twisted subrings, most of our results are framed for the more general case of subrings twisted along a multiplicatively closed set. In particular, we often assume in Sections 5--7 that the twisted subring is Noetherian. Beyond the one-dimensional case, the only ways of which  I am aware  to produce a twisted subring that is  Noetherian is  through the existence theorem (1.1) and the characterization of Noetherianness in (2.9), or using the existence result for twisted Noetherian subrings of two-dimensional UFDs that is discussed before Example~\ref{UFD example}. 

\medskip

{\it  Notation and terminology.} All rings in this article are commutative and have an identity.  When $I$ is an ideal of the ring $S$,  then $\mu_S(I)$ denotes the minimal number of elements needed to generate the ideal $I$.  If $L$ is an $S$-module,   a {\it derivation} $D:S \rightarrow L$ is an additive mapping such that $D(st) = sD(t) + tD(s)$ for all $s,t \in S$. When $C$ is a subset of $S$ such that $D(C) = 0$, then $D$ is a {\it $C$-linear derivation}. If $C$ is a multiplicatively closed subset of $S$ consisting of nonzerodivisors, then the $S$-module $L$ is {\it $C$-torsion} provided that for each $\ell \in L$, there exists $c \in C$ such that $c\ell = 0$. It is {\it $C$-torsion-free} if for all $c \in C$ and $\ell \in L$,  $c\ell = 0$ implies $\ell  =0$. The module $L$ is {\it $C$-divisible} if for each $\ell \in L$ and $c \in C$, there exists $\ell' \in L$ such that $\ell=c\ell'$.

\section{Twisted subrings}

In this section we discuss properties of twisted subrings. These subrings occur within $C$-analytic extensions, a notion we recall first. 
Let $\alpha:A \rightarrow S$ be a homomorphism of rings, and let $C$ be a multiplicatively closed subset of $A$  such that the elements of $\alpha(C)$ are nonzerodivisors of $S$.  
 Then, following \cite{Wei}, $\alpha$ is an {\it analytic isomorphism along $C$} if for each $c \in C$, the induced mapping $\alpha_c:A/cA \rightarrow S/cS:a \mapsto \alpha(a) + cS$ is an isomorphism.  When $A$ is a subring of $S$ and the mapping $\alpha$ is the inclusion mapping, we say that $A \subseteq S$ is a {\it $C$-analytic extension}.

\medskip

 {\bf (2.1)} When $A \subseteq S$ is a  ${{\it C}}$-analytic extension,  then the mappings $I \mapsto IS$ and $J \mapsto J \cap A$ yield a
one-to-one correspondence between ideals $I$ of $A$ meeting ${{\it
C}}$ and ideals $J$ of $S$ meeting ${{\it C}}$. Prime ideals of $A$
meeting ${{\it C}}$ correspond to prime ideals of $S$ meeting ${{\it
C}}$, and maximal ideals of $A$ meeting $C$ correspond to maximal
ideals of $S$ meeting $C$.
  If $J$ is a finitely generated ideal of $S$ meeting ${{\it C}}$ that can be generated
by $n$ elements, then $J \cap A$ can be generated by $n+1$ elements.
If also $A$  is quasilocal, then $J \cap A$ can be
generated by $n$ elements. All this can be found in \cite[Proposition 2.4]{OlbCounter}.

\medskip

Let $S$ be a ring, let $K$ be an $S$-module, and let $C$ be a multiplicatively closed subset of nonzerodivisors on $S$ that are also nonzerodivisors on $K$. Suppose that $R$ is twisted by $K$ along $C$, and $D$ is the derivation that twists it. 

\medskip

{\bf (2.2)} For each $s \in S$, there exists $c \in C$ such that $cs \in R$. Thus $R_C = S_C$, so that $R$ and $S$ share the same total ring of quotients \cite[Theorem 4.1]{OlbCounter}.
\medskip

{\bf (2.3)}
The extension $R \subseteq S$ is {\it quadratic}, meaning that 
 every $R$-submodule of $S$ containing $R$ is a ring; equivalently, $st \in sR + tR + R$ for all $s,t \in S$.  
   In particular, $R \subseteq S$ is an integral extension that is not finite unless $R = S$ \cite[Theorem 4.1]{OlbCounter}.   

\medskip

{\bf (2.4)} The mappings $P \mapsto PS$ and $Q \mapsto Q \cap R$ define a
one-to-one correspondence between prime ideals $P$ of $R$ meeting
${{\it C}}$ and prime ideals $Q$ of $S$ meeting ${{\it C}}$. Under
this correspondence, maximal ideals of $R$ meeting $C$ correspond to
maximal ideals of $S$ meeting $C$.
 If also $S$ is a domain, then
the contraction mapping $\Spec(S) \rightarrow \Spec(R)$ is a bijection \cite[Theorem 4.2]{OlbCounter}.

\medskip

  Let $B$ be a ring and $L$ be a $B$-module.  Then the {\it idealization} $B \star L$ of $L$ is defined as  an abelian group to be $B \oplus L$, and whose ring multiplication is given by $$(b_1,\ell_1)\cdot(b_2,\ell_2) = (b_1b_2,b_1\ell_2 +b_2\ell_1).$$  Thus the ring $B \star L$ has a square zero ideal corresponding to $L$.  

\medskip

{\bf (2.5)}  
The mapping $f:R \rightarrow S \star K:r \mapsto (r,D(r))$
 is an analytic isomorphism along $C$; see  \cite[Theorem 4.6]{OlbCounter} or \cite[Proposition 3.5]{OlbAR}.

\medskip

 {\bf (2.6)}  Let $I$ be an ideal of $R$ meeting $C$. For each $R$-module $L$, let $\widehat{L}_I$ denote the completion of $L$ in the $I$-adic topology; that is, $\widehat{L}_I = \lim_{\leftarrow} L/I^kL$. 
Then the mapping $f$ in (2.5) lifts to an isomorphism of rings, $\widehat{R}_I \rightarrow \widehat{S}_I \star \widehat{K}_I.$   When $R$ is quasilocal and $I$ is the maximal ideal of $R$, we drop the subscript $I$ in the notation $\widehat{L}_I$ and write $\widehat{L}$ for the ${\bf m}$-adic completion of $L$. In the case where $S$ is quasilocal with finitely generated maximal ideal $N$ meeting $C$, then since by (2.4), $N = MS$, it follows that the $N$-adic and $M$-adic topologies on $S$-modules agree, so the completions of $S$ and $K$ can be viewed  in either the $M$-adic or $N$-adic topologies. 

\medskip


{\bf (2.7)} 
 With $A = S \cap \Ker D$, the extension $A \subseteq S$ is  $C$-analytic.   Conversely,  if there exists a subring $A$ of $R$ containing $C$ such that $A  \subseteq S$ is $C$-analytic, $R \subseteq S$ is quadratic, and $S/R$ is $C$-torsion,  then $R$ is twisted  along $C$ by a $C$-torsion-free $S$-module
 \cite[Theorem 2.5]{OlbCounter}.
 
 \medskip

Now we strengthen our hypotheses on $R$, $S$ and $K$ and assume that $S$ is a domain with quotient field $F$, $K$ is a torsion-free $S$-module, $R$ is a subring of $S$ strongly twisted by $K$, and $D$ is the derivation that strongly twists $R$. 

\medskip

{\bf (2.8)} 
The extension $R \subseteq S$  is a subintegral extension \cite[Theorem 4.1]{OlbCounter}.

\medskip

{\bf (2.9)}
 The ring $R$ is a Noetherian domain if and only if $S$ is a Noetherian domain and for each $0 \ne a \in S \cap \Ker D$,  $K/aK$ is a finitely generated $S$-module \cite[Theorem 5.2]{OlbCounter}.

\medskip

{\bf (2.10)} The mapping $f:R \rightarrow S \star K$ in (2.5) is faithfully flat \cite[Theorem 4.6]{OlbCounter}. 
 
 \medskip

{\bf (2.11)} With $A = S \cap \Ker D$, the extension $A \subseteq S$ is {\it strongly analytic}, meaning that  $sS \cap A \ne 0$ for all $0 \ne s \in S$ and $A \subseteq S$ is a $C$-analytic extension along 
 $C = A \setminus \{0\}$.   
  Conversely,  if there exists a subring $A$ of $R$  such that $A  \subseteq S$ is strongly analytic, $R\subseteq S$ is quadratic and $R$ has quotient field $F$, then there is a torsion-free $S$-module $K$ such that  $R$ is strongly twisted   by $K$
  \cite[Corollary 2.6]{OlbCounter}.   
 
 \medskip
 
 To illustrate some basic aspects of the construction, we 
return to the example discussed in the introduction: 
 $L$ is a field of positive characteristic that is separably generated and of infinite transcendence degree over a countable subfield $k$, and  $X_1,\ldots,X_d$ are indeterminates for $L$. Let   $S$ be a local Noetherian ring between $L[X_1,\ldots,X_d]$ and $F:=L(X_1,\ldots,X_d)$,  and let $K$ be a nonzero finitely generated torsion-free  $S$-module. 
 The existence result (1.1) shows that there exists a strongly twisted subring $R$ of $S$. Thus there is a derivation $D:F \rightarrow FK$ that strongly twists $R$; in particular, $R = S \cap D^{-1}(K)$.  
However, this  derivation is constructed  in the proof of \cite[Theorem 3.5]{OlbCounter} in  an {\it ad hoc} set-theoretic way and hence it is not clear how to write  $R$ in more explicit terms. (The requirements that 
 $D(F)$ generate $FK$ as an $F$-vector space
and
  $S \subseteq  \Ker D + sS$ for all $0\ne s \in S$ are a challenge to satisfy simultaneously and rule out more ``natural'' choices of derivations.)  
 In any case, 
 from the point of view of the construction, the key property of $D$ is simply that it exists. For then $R$ is a local Noetherian ring (2.9), the extension $R \subseteq S$ is subintegral and the $M$-adic completion of $R$ (where $M$ is the maximal ideal of $R$) is isomorphic to $\widehat{S} \star \widehat{K}$, where the completions in the idealization are taken in the $N$-adic topologies with $N$ the maximal ideal of $S$.   In particular, $R$ is analytically ramified. As we show in the next sections, more subtle information about $R$ can be extracted from $S$ and $K$ using  (the existence of) the derivation $D$.    


\section{Reductions of ideals}

 Let $A$ be a ring, and let $I$ be an ideal of $A$.  An ideal $J \subseteq I$ is a {\it reduction} \index{reduction} of $I$ if there exists $n > 0$ such that $I^{n+1} = JI^n$.
The smallest integer $n$ such that $I^{n+1} = JI^n$ is the
{\it reduction number}  of
$I$ with respect to $J$.
A reduction $J$ of $I$ is a \index{reduction!minimal} {\it minimal reduction} of $I$  if $J$ itself has no proper reduction.  If $I$ is an ideal of a local Noetherian ring, then minimal reductions must exist \cite[Theorem 8.3.5]{HS}.  The \index{analytic spread} {\it analytic spread} of an ideal $I$ in a local Noetherian ring $(A,{\ff m})$, denoted $\ell_A(I)$, is useful in detecting minimal reductions, as we discuss below.   It is defined to be  the Krull dimension of the {fiber cone}  \index{fiber cone} of $I$ with respect to $A$, where for an ideal $I$ and an $A$-module $L$, the {\it fiber cone} of $
I$ with respect to $L$ is $${\cal F}_{I,L}:= L[It]/{\ff m}L[It] \cong  \bigoplus_{n=0}^\infty I^nL/{\ff m}I^{n}L.$$ Thus ${\cal F}_{I,L}$ is an $A$-module, and when $L = A$ we may view ${\cal F}_{I,L}$ as a ring in the obvious way.

We consider in this section reductions of  contracted ideals in twisted subrings.  Although the main case we have in mind is the Noetherian one, many of the results are proved in the following more general setting, which is our standing assumption for this section.

\begin{quote} {\it Let $S$ be a quasilocal ring, let $K$ be an $S$-module, and let $C$ be a multiplicatively closed subset of nonzerodivisors of $S$ that are also nonzerodivisors on $K$, and
let $R$ be a (necessarily quasilocal) subring  twisted along $C$ by a derivation $D$.} \end{quote}

In the setting of this section, the following theorem shows that  a reduction of $IS$ gives rise to a reduction of $I$, but slightly deeper inside $I$.

\begin{thm} \label{pre-construction reduction} Suppose  that $I$ is an ideal of $R$ that is contracted from $S$ and  meets $C$. \index{twisted subring!and reduction}
Let $n>0$.  If there exists a finitely generated ideal $J$ of $S$ with $$J \subseteq IS \: \:  {\mbox{ and }} \:
 (IS)^{n+1} = J(IS)^n,$$ then there exists an $R$-ideal $J' \subseteq I$  with $ I^{n+2} = J'I^{n+1}$  and $\mu_R(J') = \mu_S(J)$.


\end{thm}

\begin{proof} Let $A = C \cap \Ker D$, where $D$ is the derivation that twists $R$. By (2.7), $A \subseteq S$ is a $C$-analytic extension.  First we claim that $I^2 = (I \cap A)I$.  Let $x,y \in I$.  We show that $xy \in (I \cap A)I$.   Let $c \in I \cap C$.   Then since $S/A$ is $C$-divisible, there exist $a,b \in A$ and $s,t \in S$ such that $x = a + cs$ and $y  = b + ct$.  Therefore, $$xy = ab + bcs + act + c^2st.$$  We examine each component of this sum separately.  First, $$a  = x  - cs \in IS \cap A = (IS \cap R) \cap A = I \cap A.$$ Similarly, $b \in I \cap A$, and thus $ab \in (I \cap A)^2$.      Next, since $I$ is contracted from $R$, we have
 $cs \in cS \cap R \subseteq IS \cap R = I$, and similarly, $ct \in I$.  Thus $bcs + act \in (I \cap A)I$. 
 Also, since $R \subseteq S$ is a quadratic extension,
$st \in sR + tR +R$, so  
  $c^2st \in c(cs)R+ c(ct)R +c^2R \subseteq cI \subseteq  (I \cap A)I$.   This proves that $I^2 = (I \cap A)I$.

We claim next that if $J_1,\ldots,J_n$ are ideals of $S$ with $c$ in each $J_i$, then $$(J_1J_2 \cdots J_n) \cap A = (J_1 \cap A)(J_2 \cap A) \cdots (J_n \cap A).$$  
From $S = A +cS$, it follows that  $J_i = (J_i \cap A)S$ and  
 $S = A+c^nS$.  Thus, multiplying both sides of the latter equality of rings by the product of the $J_i \cap A$ yields,
    $$J_1\cdots J_n  = (J_1 \cap A) \cdots (J_n \cap A) + c^nS.$$  Intersecting with $A$ and applying the Modular Law, we have: $$(J_1 \cdots J_n) \cap A = (J_1 \cap A) \cdots (J_n \cap A) + c^nS \cap A.$$ Now since $A \cap cS = cA$, 
    it follows that $A \cap c^nS = c^nA$.  Hence,
 since $c$ is in each $J_i$, the claim is proved.

Now suppose $J$ is a finitely generated ideal of $S$ such that $J \subseteq IS$ and   $(IS)^{n+1} = J(IS)^n$. Define $J' = (J \cap A)R$.  Since $I = IS \cap R$, it follows that $IS \cap A = I \cap A$. Also, since $(IS)^{n+1} \subseteq J$, then $J$ meets $C$. Thus by 
 the above claim:
$$(I \cap A)^{n+1} = (IS \cap A)^{n+1} = I^{n+1}S \cap A = JI^nS \cap A = (J \cap A)(I \cap A)^{n}.$$  Consequently,  since $I^2 =
   (I \cap A)I$, it follows  that $$I^{n+2} =  I(I \cap A)^{n+1} = (J \cap A)I(I \cap A)^n =
   (J \cap A)I^{n+1}.$$  Set $J' = (J \cap A)R$, so that $I^{n+2} = J'I^{n+1}$, and note also that $J' = (J \cap A)R \subseteq IS \cap R = I$.

 We examine more closely the possible choices of generators for  $J'$.  Write $J = (x_1,\ldots,x_k)S$, where $k = \mu_S(J)$,  and choose $c \in J \cap C$.
Then since $S = A + c^2S$, there exist $a_1,\ldots,a_k \in A$ and $s_1,\ldots,s_k \in S$ such that for each $i$, $x_i = a_i + c^2s_i$.  Note that $a_i = x_i - c^2s_i \in J \cap A$, so necessarily,  $$J = (x_1,\ldots,x_k)S = (a_1,\ldots,a_k,c^2)S.$$ By (2.1), $J \cap A = (a_1,\ldots,a_k,c^2)A$.
We claim that in fact $J \cap A = (a_1,\ldots,a_k)A$.
To this end, we first observe that $A$ is quasilocal.  For let ${\ff m}$ be the contraction of the maximal ideal of $S$ to $A$.  If $b \in A \setminus {\ff m}$, then since $S$ is quasilocal, there exists $s$ in $S$ such that $1= bs$, and hence $0 = D(1) =  D(bs) = bD(s) + sD(b) = bD(s)$.  But  then $0 = sbD(s) = D(s)$,  which means that $s \in S \cap \Ker D = A$, proving that $b$ is a unit in $A$, and hence that $A$ is quasilocal with maximal ideal ${\ff m}$.
Therefore, since $c^2 \in {\ff m}(J \cap A)$, Nakayama's Lemma implies that $J \cap A = (a_1,\ldots,a_k)A$.
Moreover, by (2.1), $(J \cap A)S = J$, so $k = \mu_S(J) \leq \mu_A(J \cap A) \leq k$. Thus 
 $\mu_A(J \cap A) = \mu_S(J)=k$, which since $J' = (J \cap A)R$ implies $\mu_R(J') \leq k$.  But $J'S = (J \cap A)S = J$,  so since  $\mu_S(J) = k$, this forces $\mu_R(J') = k$.
\end{proof}



While the next theorem seems  natural in light of the isomorphism $\widehat{R} \rightarrow \widehat{S} \star \widehat{K}$, because this isomorphism is not an $R$-algebra map with respect to the obvious $R$-module structure on $\widehat{S}\oplus  \widehat{K}$, the proof contains several subtleties, one of the main ones being that the isomorphism of rings given in the theorem contains a twist of degree $-1$.  

\begin{thm} \label{pre-construction fiber} 
Suppose that  $I$ and $J$ are proper  ideals of $R$ contracted from $S$ and meeting $C$, and such that $I \subseteq J$.  Then there is an isomorphism of rings:  \index{twisted subring!fiber cone of}
$$R[It]/JR[It] \cong S[It]/JS[It] \star K[It]/JK[It].$$  In particular,  the fiber cone of $I$ is
${\cal F}_{I,R} \cong {\cal F}_{IS,S} \star {\cal F}_{IS,K}.$
\end{thm}

\begin{proof}
Define  a mapping
 $$\beta:R[It]/JR[It] \rightarrow S[It]/JS[It] \star K[It]/JK[It]$$  by $$\beta\left(\sum_{0\leq k}i_kt^k + JR[It]\right) = \left(\sum_{0\leq k}i_kt^k + JS[It], \sum_{0 \leq k}D(i_{k+1})t^k + JK[It] \right),$$ 
 where each $i_k \in I^k$ and all but finitely many $i_k$ are $0$.  
  We claim that $\beta$ is a ring isomorphism.
 To see first that $\beta$ is well-defined, note that  $D(i_{k+1}) \in D(I^{k+1}) \subseteq I^kD(I) \subseteq I^kK$, so that $D(i_{k+1})t^k \in K[It]$. Moreover,  
  suppose that $\sum_{0\leq k} i_kt^k \in JR[It]$, where for each $k$, $i_k \in JI^k$, and all but finitely many of the $i_k$ are $0$.  
  It suffices to show that $\sum_{0 \leq k}D(i_{k+1})t^k \in JK[It]$, and to prove this it is enough to check that there is a containment, $\sum_{0\leq k}D(JI^{k+1})t^k \subseteq JK[It]$.   
This is indeed the case, since 
  \begin{eqnarray*}
 \sum_{0 \leq k} D(JI^{k+1})t^k  
 & \subseteq &   \sum_{0 \leq k} (JI^{k}D(I) + I^{k+1}D(J))t^k  \\
\: & \subseteq &  \sum_{0 \leq k} JI^kKt^k \:\: \subseteq \:\:  JK[It].
\end{eqnarray*}
This shows that $\beta$ is well-defined.

We claim next that $\beta$ is a ring homomorphism.  It is clearly additive and unital.  To see that $\beta$ preserves multiplication, let $f_1 = \sum_{0\leq k}i_kt^k \in R[It]$ and $f_2=\sum_{0\leq k}j_kt^k \in R[It]$, and
 set $$f = f_1 \cdot f_2 = \sum_{0 \leq k} \:\sum_{n+m=k}  i_n j_m t^k.$$ Then properties of derivations show 
 %
 %
\begin{eqnarray*}
\beta(f+J[It]) & = & \left( f + JS[It], \:   \sum_{0 \leq k} D\left( \sum_{n+m=k+1}  i_n j_m\right) t^k +JK[It]\right) \\
\: & = & \left(  f+JS[It], \: \sum_{0 \leq k} \: \sum_{n+m = k+1} (i_nD(j_m)+j_mD(i_n))t^k + JK[It] \right) \\
\: & = & \beta\left(  f_1+J[It] \right) \cdot \beta\left(f_2+ J[It] \right).
\end{eqnarray*} Therefore, $\beta$ is a ring homomorphism.

To see that $\beta$ is one-to-one, let $\sum_{0\leq k}i_kt^k \in R[It]$, so that each $i_k \in I^k$, and suppose that $$\sum_{0\leq k}i_kt^k \in JS[It] \:  {\mbox{ and }} \: \sum_{0 \leq k}D(i_{k+1})t^k \in JK[It].$$  Then for each $k\geq 0$, we have
 $$i_k \in I^k, \: i_k \in JI^kS \: {\mbox{ and }} \: D(i_{k+1}) \in JI^{k}K.$$
We claim that for each $k\geq 0$, $i_k \in JI^k.$  We have  $i_0 \in JS \cap R = J$, so the claim is true for $k=0$.  Now suppose $k>0$.
Then by assumption, $$i_k \in I^k, \: i_k \in JI^kS \: {\mbox{ and }} \: D(i_{k}) \in JI^{k-1}K.$$
Choose $c \in JI^k \cap C$. Let $A = S \cap \Ker D$, so that by (2.7), $A \subseteq S$ is a $C$-analytic extension.   By (2.1), since $J = JS \cap R$ and $I = IS \cap R$, it follows that 
 $JS = (J \cap A)S$ and $IS = (I \cap A)S $.      Set ${\ff a} = (J \cap A)(I \cap A)^k$.    Since $S = A + cS$, we have
$$JI^kS   =   {\ff a}S =  {\ff a} + c{\ff a}S \subseteq {\ff a} + cS.$$
Thus we may write $i_k = a + c\sigma$, for some $a \in {\ff a}$ and  $\sigma \in S$.
 Since $a \in {\ff a} \subseteq JI^k$, to show that $i_k \in JI^k$, it suffices to prove that $c\sigma \in   JI^k$.
Let ${\ff b} = (J \cap A)(I \cap A)^{k-1}$.  As above, ${\ff b}S = JI^{k-1}S$.
Now since $D(A)=0$ and  $K = D(cS \cap R) + cK$ (see the proof of \cite[Proposition 3.5]{OlbAR}), then
\begin{eqnarray*}
D(c\sigma)  & = & D(a+c\sigma) \:\: = \:\: D(i_k) \\
& \in & JI^{k-1} K  \:\: =  \:\:{\ff b}K
\:\: =  \:\:  {\ff b}(D(cS \cap R) + cK) \\
\: & \subseteq & D({\ff b}(cS \cap R)) + cK.
\end{eqnarray*}
So there exist $a_1,\ldots,a_m \in {\ff b}$ and $\sigma_1,\ldots,\sigma_m \in S$ such that $c\sigma_j \in R$ for all $j$ and
 $$D(c\sigma) - D(\sum_{j}a_jc\sigma_j ) \in  cK.$$ Hence, since  $D^{-1}(cK) \cap cS = cR$ (see the proof of \cite[Proposition 3.5]{OlbAR}), then
 $$c\sigma -  \sum_{j}a_jc\sigma_j \in D^{-1}(cK) \cap cS = cR.$$  Now each $a_j \in JI^{k-1}$, and $c\sigma_j \in cS \cap R \subseteq IS \cap R = I$, so each $a_jc\sigma_j \in JI^k$.   Thus since $c \in JI^k$, we have $c\sigma \in JI^k$, as claimed.  This proves that $\beta$ is one-to-one.

Next to prove that $\beta$ is onto, it suffices since $\beta$ is an additive mapping to verify the following two claims.

\smallskip

{\noindent}(i)  For 
each $i_k \in I^kS$, there exists $b \in I^k$ such that $$\beta(bt^k + JK[It]) = (i_kt^k + JS[It], 0 + JK[It]).$$ 

\smallskip

{\noindent}(ii) For each  $i_k \in I^k$ and $y \in K$ there exists 
$r \in I$ such that $$\beta(ri_kt^{k+1}+ JR[It]) = (0 + JS[It], \: i_kyt^k + JK[It]).$$

To verify (i), suppose
$i_k \in I^kS$, and let $c \in JI^k \cap C$.  As above, (2.1) implies $I^kS = (I \cap A)^kS$, so that
 since $S = A + cS$, we have $I^kS = (I \cap A)^k + c(I \cap A)^kS$.  Hence we
 may write  $i_k = b + cs$, with  $b \in (I \cap A)^k$ and $s \in S$.   Then $cs \in JI^kS$, so that $cst^k \in JS[It]$.  Moreover, $bt^k \in R[It]$, and since $b \in A$, $D(b) =0$.
Thus if $k = 0$, then $$\beta(b + JR[It]) = (i_0+JS[It], 0 + JK[It]),$$ which verifies (i) in the case $k=0$.  Otherwise, if $k>0$, then:
\begin{eqnarray*}
\beta( bt^k + JR[It] ) & = &  \left(bt^k +JS[It],\: D(b)t^{k-1} + JK[It]  \right) \\
\: & = & (i_kt^k+JS[It], 0 + JK[It])
\end{eqnarray*}
This proves claim (i).  

Next, to verify claim (ii), suppose that $i_k \in I^k$ and $y \in K$.  
Choose $c \in JI^{k+1} \cap C$.
As noted above, $K = D(cS \cap R) + cK$, so we may write $y = D(r) + cw$ for some $w \in K$ and $r \in cS \cap R$.  Note that since $c \in I$,  and $I = IS \cap R$, we have $r \in cS \cap R \subseteq IS \cap R = I$, so that $ri_k \in I^{k+1}$.
Also, since $c \in JI^{k+1}$, we have $r \in cS \subseteq JI^{k+1}S$, so that $ri_k \in JI^{k+1}S$.
Therefore, $ri_kt^{k+1} \in R[It] \cap JS[It]$, while $r \in  JI^kS$, and applying these facts we have:
\begin{eqnarray*}
\beta( ri_k t^{k+1} + JR[It] )&  = &
(r i_k t^{k+1} + JS[It], \: D(ri_k)t^k + JK[It]) \\
\: & = & (0+ JS[It], \: rD(i_k)t^k + i_kD(r)t^k + JK[It]) \\
\: & = & (0+ JS[It], \:  i_kD(r)t^k + JK[It]) \\
\: & = & (0 + JS[It], \: i_k(y - cw)t^k + JK[It]) \\
\: & = & (0 + JS[It], \: i_kyt^k + JK[It]).
\end{eqnarray*}  This proves claim (ii), and so we conclude that $\beta$ is onto, and hence an isomorphism.

To prove the last assertion, let $M$ and $N$ denote the maximal ideals of $R$ and $S$, respectively.  Then since by (2.4),  $N = MS$, by setting $J = M$, we obtain:  $${\cal F}_{I,R} = R[It]/MR[It] \cong S[It]/NS[It] \star K[It]/NK[It]= {\cal F}_{IS,S} \star {\cal F}_{IS,K},$$ and this proves the theorem.
\end{proof}

Thus for a contracted ideal $I$ meeting $C$, the Krull dimension of ${\cal F}_{I,R}$ is the same as the Krull dimension of ${\cal F}_{IS,S}$.  This translates into the following statement about analytic spread.

\begin{cor} \label{analytic spread} For each ideal $I$ of $R$ contracted from $S$ and meeting $C$, $\ell_R(I) = \ell_S(IS).$  If $R$ is strongly twisted by $K$, then this holds for all nonzero ideals of $R$ contracted from $S$. \qed\end{cor}

To each ideal $I$ of a local Noetherian ring $A$, there is associated the invariant $r_A(I)$, \index{$r_A(I)$} which is defined to be the minimum of the reduction numbers of the minimal reductions of $I$. From the theorem and facts about minimal reductions, we deduce that the reduction number of a contracted ideal of $R$ is at most one more than its extension in $S$:

\begin{cor}\label{minimal red} Suppose in addition that $R$ and $S$ are  local Noetherian rings with  infinite residue field.  For each  ideal $I$ of $R$ contracted from $S$ and meeting $C$, if $J$ is a minimal reduction of $IS$ such that $I^{n+1}S = JI^nS$,  then there exists a minimal reduction $J'$ of $I$ such that $I^{n+2} = J'I^{n+1}$  and  $$r_S(IS) \leq r_R(I) \leq r_S(IS)+1.$$
\end{cor}

\begin{proof} 
If $E$ is a reduction of the ideal $I$ with $\mu(E) = \ell(I)$, then $E$ is necessarily a minimal reduction of $I$ \cite[Corollary 8.3.6]{HS}.  The converse is true since the residue field of the local ring is infinite; that is, if $E$ is a minimal reduction of $I$, then $\mu(E) = \ell(I)$ \cite[Proposition 8.3.7]{HS}.
Applying these facts in our setting,
 suppose that $J$ is a minimal reduction of $IS$ 
such that 
  $I^{n+1}S = JI^nS$ and $J$ is 
 chosen so that $n = r_S(IS)$.  
   Note that since $I$ meets $C$ and $J$ contains a power of $I$, then $J$ also meets $C$.  
  Now $\mu_S(J) = \ell_S(IS)$, and by Theorem~\ref{pre-construction reduction}, there exists a reduction $J'$ of $I$ with $\mu_R(J') = \ell_S(IS)$ and $I^{n+2} = J'I^{n+1}$.  But by Corollary~\ref{analytic spread}, $\ell_R(I) = \ell_S(IS) = \mu_R(J')$, so necessarily $J'$ is a minimal reduction of $I$.  
%
Therefore, $r_R(I) \leq r_S(IS) + 1$.    On the other hand, if $E$ is a minimal reduction of $I$, then $\mu_S(ES) \leq \mu_R(E) = \ell_R(I) = \ell_S(IS)$. But since $ES$ is a reduction of $IS$, then $\ell_S(IS) \leq \mu_S(ES)$. Thus $\mu_S(ES) = \ell_S(IS)$, so that $ES$ is a minimal reduction of $IS$.  Consequently,  $r_{S}(IS) \leq  r_R(I)$, which proves $r_S(IS) \leq r_R(I) \leq r_S(IS)+1$.
\end{proof}

Both of the cases 
  $r_{S}(IS) =  r_R(I)$ and $r_R(I) = r_S(IS)+1$
in the corollary can occur.  The next example illustrates the second case.  The first case, where   $r_S(IS)=r_R(I)$, is treated later in Example~\ref{second reduction example} after  minimal multiplicity has been considered.

\begin{exmp} \label{first reduction example}
{\em It can happen that $r_R(I) = r_S(IS) + 1$.  Using (1.1), it is easy to arrange for  $S$ to be a regular local ring with infinite residue field.  Then choosing $K$ to be a nonzero finitely generated torsion-free $S$-module, we obtain  a local Noetherian subring $R$ of $S$ strongly twisted by $K$.  Moreover, since $K$ is finitely generated, (2.9) shows that $R$ is a local Noetherian domain.  Since $S$ is regular, the maximal ideal $N=MS$  of $S$ (see (2.4)) is its own minimal reduction, so $r_S(MS) = 0$.  But if $r_R(M) = 0$, then $M$ is its own minimal reduction, and hence by Corollary~\ref{analytic spread}, $\mu_R(M) = \ell_R(M) = \ell_S(MS) = \mu_S(MS)$.  This then forces $M$ to be generated by the same number of elements that minimally generate $N=MS$, which means that $R$ is also a regular local ring.  Yet $R \subsetneq S$ is integral and $R$ and $S$ share the same quotient field, so this is impossible.  So necessarily $r_R(M) = 1$ while $r_S(MS) = 0$. \qed}
\end{exmp}


There is  a special case in which it is possible to be more definitive.  Applying Corollary~\ref{minimal red} to the case where the extension of $I$ in $S$ is a principal ideal,
we have:

\begin{cor} If $I$ is an ideal of $R$ contracted from a principal ideal of $S$ meeting $C$, then $r_R(I) = 1$ and  $I^2 = iI$ for some $i \in I$.
\end{cor}

\begin{proof} By Theorem~\ref{pre-construction reduction}, we have $I^2 = iI$ for some $i \in I$, and as we see in Theorem~\ref{pre-construction Hilbert} below, $I$  needs at least one more generator  than the number of generators that the $S$-module  $K$ requires. Thus $I$ is not a principal ideal, and it follows that $iR$ is a proper minimal reduction of $I$. Hence  $r_R(I) = 1$. 
  \end{proof}

The fact that $I^2 = iI$ for   a nonzerodivisor $i \in I$ implies  that $I$ is {\it stable},
 meaning that $I$ is a projective module over its ring of endomorphisms; see \cite{OlbCounter, OlbAR} for more on the connection between stable ideals and twisted subrings.

\section{Hilbert function of a contracted ideal} 

In this section we make the following assumption:

\begin{quote} {\it $R$ and $S$ are  quasilocal rings with finitely generated maximal ideals $M$ and $N$, respectively; $C$ is a multiplicatively closed subset of nonzerodivisors of $S$ meeting the maximal ideal of $S$; $R$ is twisted along $C$ by a module $K$, where the elements of $C$ are   
 nonzerodivisors on $K$, and  $D$ is the derivation that twists $R$.} \end{quote} 

Since by (2.3), $R \subseteq S$ is integral, it follows that the maximal ideal $M$ of $R$ meets $C$. 
As the analytic isomorphism in (2.4) suggests, the Hilbert function of an ideal $I$ of $R$ should be similar to the Hilbert function of the ideal $IS$ of $S$, but offset by the Hilbert function of $K$.
Similarly, arithmetical invariants such as embedding dimension and multiplicity are affected also by $K$.
  We discuss this next by first reviewing the notion of the Hilbert function; see
 \cite{Eis}, \cite{Ma}      or \cite{Na} for more background.

 To define the Hilbert function, let $A$ be a local Noetherian ring with maximal ideal ${\ff m}$,
let $J$ be an ${\ff m}$-primary ideal of $A$, and let $L$ be an $A$-module
 such that $L/{\ff m}L$ is a finitely generated $R$-module.  (Here  we differ slightly from the traditional definition, since we do not require $L$ itself to be finitely generated.)    It follows then that $L/{\ff m}^iL$ has finite length for each $i>0$, so that since $J$ is ${\ff m}$-primary, $J^iL/J^{i+1}L$  has finite length for each $i>0$.
The \index{$H_{J,L}(n)$}
{\it Hilbert function} of $L$ with respect to $J$, denoted $H_{J,L}$, is given  by   $$H_{J,L}(n) = {\mbox{ length }}
J^{n}L/J^{n+1}L,$$ with the  convention $J^0 = A$.

 We  wish to calculate the Hilbert function when $J = M$ and $L = R$, since this then will lead to  embedding dimension and multiplicity for the ring $R$.  These two invariants of the ring  are preserved under completion, so embedding dimension and multiplicity of the ring can be calculated directly from the facts that $\widehat{R}$ and $\widehat{S} \star \widehat{K}$ are isomorphic as rings and the maximal ideal of $\widehat{R}$ is sent by this isomorphism to the maximal ideal of $\widehat{S} \star \widehat{K}$.  However, this isomorphism is not an isomorphism of $R$-algebras, at least not with the obvious $R$-module structure on $\widehat{S} \star \widehat{K}$, and we wish to calculate Hilbert functions of contracted ideals of $R$, not just the maximal ideal of $R$. For this reason we use a more careful approach  which relies on the following lemma. Recall the ring homomorphism, 
$f:R \rightarrow S \star K:r \mapsto (r,D(r))$ from (2.4).

\begin{lem} \label{product lemma}
 If $I_1,\ldots,I_n$ are ideals of $R$ such that each
$I_i$ meets $C$ and $I_i = I_iS \cap R$, then $f(I_1)(S\star K)  
= (I_1S) \star K$ and  $$f(I_1 \cdots I_n)(S\star K)  = (I_1\cdots
I_nS) \: \star \: \left(\sum_{j_1 < \cdots < j_{n-1}}I_{j_1} \cdots
I_{j_{n-1}}K\right),$$ where $j_1 < \cdots < j_{n-1}$ range over the
members of the set $\{1,2,\ldots,n\}$.
\end{lem}

\begin{proof}  
It is proved in \cite[Lemma 4.5]{OlbCounter} that for each $I$ of $R$ that meets $C$ and satisfies $I = IS \cap R$, we have
 $f(I) (S \star K) = (IS) \star K$.  
We prove the general formula by induction.  Suppose that $n>0$
and that $$(I_1 \cdots I_{n-1}) \cdot (S \star K) = (I_1\cdots
I_{n-1}S) \star \left(\sum_{j_1 < \cdots < j_{n-2}}I_{j_1} \cdots
I_{j_{n-2}}K\right).$$  Then, using the fact established above that $I_n \cdot (S \star K) = (I_nS) \star K$, we have:
\begin{eqnarray*}
(I_1\cdots I_n) \cdot (S \star K) & = & \left((I_1\cdots I_{n-1}) \cdot
(S \star K)\right)\:\left(I_n \cdot (S \star K)\right)
\\
\: & = & \left( I_1\cdots I_{n-1}S \: \:  \star \: \sum_{j_1 < \cdots <
j_{n-2}}I_{j_1} \cdots I_{j_{n-2}}K\right) \: ((I_nS) \star K) \\
\: & = & (I_1\cdots I_nS) \: \star \: \left((I_1 \cdots I_{n-1} +
\sum_{j_1 <
\cdots < j_{n-2}} I_{j_1} \cdots  I_{j_{n-2}}I_n)K\right) \\
\: & = & (I_1 \cdots I_nS) \: \star \: \left( \sum_{j_1 < \cdots <
j_{n-1}}I_{j_1} \cdots I_{j_{n-1}}K\right).
\end{eqnarray*}
\end{proof}

If $R$ is a Noetherian ring, then since $N$ meets $C$, it follows from (2.9) that $K/NK$ is a finitely generated $S$-module.  Thus the following theorem applies in the main case of interest, that in which $R$ is Noetherian.

\begin{thm} \label{pre-construction Hilbert} Suppose  that $K/NK$ is a finitely generated $S$-module. 
      Then the following statements hold.

\begin{itemize}


\index{twisted subring!and Hilbert function}
\item[{(1)}]  For nonzero finitely generated ideals $I$ and $J$ of $R$ contracted from $S$ and meeting $C$,  where $J$ is $M$-primary, we have for each $n>0$,
$$H_{J,I}(n) = H_{JS,IS}(n) + H_{JS,(J+I)K}(n-1).$$  If also $I \subseteq J$, then we have for each $n\geq 0$, $$H_{J,I}(n) = H_{JS,IS}(n) + H_{JS,K}(n).$$

\item[{(2)}]  For each proper  nonzero finitely generated ideal $I$ of $R$ contracted from $S$ and meeting $C$, the minimal number of generators of $I$ is $$\mu_R(I) = \mu_S(IS)+ \dim_{S/N}K/NK.$$

 \index{twisted subring!and minimal number of generators}
\end{itemize}
\end{thm}

\begin{proof}
 Let  $A = S \cap \Ker D$, and note that by (2.7), $A \subseteq S$ is a $C$-analytic extension.

(1)   Let  $I$ and $J$ be ideals of $R$ which are contracted from $S$ and   meet $C$, where $J$ is $M$-primary.  We claim first that  for all $i >
0$ there is an isomorphism of $R$-modules (where the $R$-module structure on the direct sum is the usual one induced by the $R$-module structures on the components): \begin{eqnarray}\label{dagger} J^iI/J^{i+1}I \cong
\left( J^iIS/J^{i+1}IS \right) \oplus \left( J^{i-1}(J+I)K/J^{i}(J+I)K \right), \end{eqnarray}
 and that when $I \subseteq J$, there is for $i \geq 0$ an isomorphism:  \begin{eqnarray}\label{dagger2}J^iI/J^{i+1}I \cong
\left( J^iIS/J^{i+1}IS \right) \oplus \left( J^{i}K/J^{i+1}K \right).\end{eqnarray}
Note that for $i>0$ and $I \subseteq J$ the formula (\ref{dagger}) coincides with (\ref{dagger2}), so when $I \subseteq J$, it is only the case $i = 0$ that concerns us when distinguishing between the two formulas.

Let $f$ denote the analytic isomorphism $f:R \rightarrow S \star K:r \mapsto (r,D(r))$ given by (2.4). 
 Since both $I$ and $J$ are contracted from $S$, we may apply the product formula in Lemma~\ref{product lemma} to obtain for each $i > 0$ that: \begin{eqnarray} \label{dagger3} f(J^iI)  (S \star
K) = J^iIS \star (J^i + J^{i-1}I)K = J^iIS \star J^{i-1}(J+I)K.\end{eqnarray}
Moreover, if $I \subseteq J$ and $i>0$, this formula yields:
\begin{eqnarray}\label{dagger4} f(J^iI)  (S \star
K) = J^iIS \star J^{i-1}(J+I)K = J^iIS \star J^{i}K.\end{eqnarray} In the case $i =0$, we have by Lemma~\ref{product lemma}, \begin{eqnarray}\label{dagger5}
f(I)  (S \star K) = (IS) \star K.
\end{eqnarray}

 Let $i > 0$, and choose $0 \ne c \in
J^{i}(J+I) \cap C$.
By (2.4), the mapping $g:R/cR \rightarrow S/cS \star
K/cK$ induced by $f$ is an isomorphism of rings.  Therefore, from (\ref{dagger3}) we have
$$g(J^iI/cR) = g(J^iI/cR) \cdot (S/cS \star K/cK) = J^iIS/cS \star
J^{i-1}(J+I)K/cK$$ and $$g(J^{i+1}I/cR) = g(J^{i+1}I/cR) \cdot  (S/cS \star K/cK) =
J^{i+1}IS/cS \star J^{i}(J+I)K/cK.$$ Thus since $g$ is an isomorphism,
there is an isomorphism of $A$-modules:
\begin{eqnarray*}
J^iI/J^{i+1}I  & \cong & (J^{i}I/cR)/(J^{i+1}I/cR) \\
\: & \cong &
(J^iIS/J^{i+1}IS) \oplus (J^{i-1}(J+I)K/J^{i}(J+I)K).
\end{eqnarray*}  The second isomorphism  is also an isomorphism of $R$-modules, where the $R$-module structure is the usual one on the direct sum.
Indeed, since $S = A + cS$, we have $R = A + cS \cap R$.  Moreover, $c \in J$, and $J = JS \cap R$,  so $R = A + cS \cap R = A + J$, and from this it follows that this isomorphism of $A$-modules is also an
isomorphism of $R$-modules.  This verifies (\ref{dagger}).

Now a similar argument using (\ref{dagger4}) and (\ref{dagger5}) instead of (\ref{dagger3})  shows that when $I \subseteq J$ there is an isomorphism of $R$-modules for all $i\geq 0$,
$$J^iI/J^{i+1}I  \cong
(J^iIS/J^{i+1}IS) \oplus (J^{i}K/J^{i+1}K),$$   and this verifies (\ref{dagger2}).

Now to prove statement (1) of the theorem,
we make use of the isomorphisms (\ref{dagger}) and (\ref{dagger2}).  We consider the case first where we make no assumption about whether $I \subseteq J$.
Since $S = R + JS$,  it follows that the
lengths of the $S/JS$-modules,
 $$J^iIS/J^{i+1}IS \: {\mbox{ and }} \: J^{i-1}(I+J)K/J^i(I+J)K,$$  agree with  their lengths as $R/J$-modules.
Therefore, we have by (\ref{dagger}) that for all $n\geq 1$,
\begin{eqnarray*}
H_{J,I}(n) & = & \length J^nI/J^{n+1}I \\
\: & = &  \length J^nIS/J^{n+1}IS + \length J^{n-1}(J+I)K/J^n(J+I)K \\
\: & = & H_{JS,IS}(n) + H_{JS,(I+J)K}(n-1).
\end{eqnarray*}
Next consider the case where $I \subseteq J$.  Then by (\ref{dagger2}) we have that for all $n\geq 0$,
\begin{eqnarray*}
H_{J,I}(n) & = & \length J^nI/J^{n+1}I \\
\: & = &  \length J^nIS/J^{n+1}IS + \length J^{n}K/J^{n+1}K \\
\: & = & H_{JS,IS}(n) + H_{JS,K}(n).
\end{eqnarray*}

(2) By assumption  $I \subseteq M$, so we let $M$ play the role of $J$ in statement (1) of the theorem.    Since by (2.4),  $N = MS$, we have by the second expression for the Hilbert function in (1), the expression which admits $n=0$ as a possible input, that
\begin{eqnarray*}
\mu_R(I) &  = &  \dim_{R/M} I/MI \:\: = \:\: H_{M,I}(0) \\
\: &= & H_{N,IS}(0) + H_{N,K}(0) \\
\: & = & \dim_{S/N} IS/NIS + \dim_{S/N} K/NK \\
\: &  =& \mu_{S}(IS) + \dim_{S/N} K/NK.
\end{eqnarray*}
\end{proof}

\section{Local cohomology} 

  Let $A$ be a Noetherian ring, let $I$ be an ideal of $A$ and let $L$ be an $A$-module.  Following \cite[Sections 1.1 and 1.2]{BrSh}, the {\it $I$-torsion functor} $\Gamma_I$ is defined on $L$ by $$\Gamma_I(L) = \bigcup_{n>0}(0:_L I^n).$$  This functor is left exact, and its $i$-th derived functor, denoted $H_I^i$, is the {\it $i$-th local cohomology functor of $L$ with respect to $I$}. \index{local cohomology functor} For our purposes here, the following interpretation of $H_I^i$ is the most useful: $$H_I^i(L) \cong \lim_{\rightarrow} \Ext_A^i(A/I^k,L),$$ where the direct limit ranges over $k \in {\mathbb{N}}$ \cite[Theorem 1.3.8, p.~15]{BrSh}.



We show in this section that when $R$ is a twisted subring of $S$, 
the local cohomology modules $H_I^i(R)$ are determined  by $S$ and the module $K$ that twists $R$.  In the case where $I$ is the maximal ideal of $R$ and  $K$ is a finitely generated $S$-module, then since $\widehat{R}$ and $\widehat{S} \star \widehat{K}$ are isomorphic as rings, a standard argument such as that in \cite[Lemma 5.1]{Val} would be sufficient to verify Theorem~\ref{local cohomology}.  But since   we wish to calculate local cohomology of contracted ideals in general, and the preceding isomorphism is of rings, not necessarily $R$-algebras, the situation is more nuanced.

The following lemma  is a variation on a well-known fact about local cohomology of completions.  When $L$ is a finitely generated $A$-module, then the lemma is valid for any $J$-adic completion, where $J$ is an ideal of $A$ with $J \subseteq I$; see \cite[Proposition 5.9]{Hart}.  Where our lemma differs is that we do not assume $L$ is a finitely generated $A$-module, a generality that is needed in Section~\ref{V section}.   However, our arguments require us to  restrict our choice of $J$ to principal ideals.

\begin{lem} \label{Groth} Let $A$ be a Noetherian ring, and let $L$ be an $A$-module.  Suppose that $c \in A$ is a nonzerodivisor that is also a nonzerodivisor on $L$. Let $A' = \lim_{\leftarrow} A/c^kA$ and $L' = \lim_{\leftarrow} L/c^kL$ be the $cA$-adic completions of $A$ and $L$, respectively.  If $I$ is an ideal of $A$ containing $c$, then for each $i \geq 0$, there is an isomorphism of $A$-modules, $$H^i_I(L) \cong H^i_{IA'}(L').$$
\end{lem}

\begin{proof}  Let $C = \{c^k:k>0\}$.
We claim first that for all $i,k\geq 0$, $\Ext_A^i(A/I^k,-)$  vanishes on $C$-divisible $C$-torsion-free $A$-modules. Let $K$ be a $C$-divisible $C$-torsion-free $A$-module. 
 Since $cK = K$ and $c$ is a nonzerodivisor on $K$, then $K$ is an $A_C$-module.  Fix $i,k\geq 0$.  Then since $A_C$ is a flat $A$-module, we may change rings \cite[Chapter VI, Proposition 4.1.3]{CE}: $$\Ext^i_A(A/I^k,K) \cong \Ext^i_{A_C}((A/I^k)\otimes_A A_C, K).$$  By assumption $c^k \in I^k$, so $(A/I^k) \otimes_A A_C  = 0$, and hence $\Ext_A^i(A/I^k,K) = 0$. Thus $\Ext_A^i(A/I^k,-)$  vanishes on $C$-divisible $C$-torsion-free $A$-modules.

 Let
 $D = \bigcap_{k\geq 0}c^kL$ and
let $L_0 = L/D$.  Then $L_0$ embeds in $L'$ and we  identify $L_0$ with its image in $L'$ (which is the same as the image of $L$ in $L'$).   It is straightforward to check that  $L'/L_0$ is a $C$-torsion-free $C$-divisible $A$-module. 
Since by what we have established above, $\Ext_A^i(A/I^k,L'/L_0) = 0$ for all $i \geq 0$,  consideration of the long exact sequence,
  $$\cdots \rightarrow \Ext_A^i(A/I^k,L_0) \rightarrow \Ext_A^i(A/I^k,L') \rightarrow \Ext_A^i(A/I^k,L'/L_0) \rightarrow \cdots$$ shows that for all $i \geq 0$, the mapping $\Ext_A^i(A/I^k,L_0) \rightarrow \Ext_A^i(A/I^k,L')$ is an isomorphism of $A$-modules.  Now consider
 the long exact sequence,
 $$\cdots \rightarrow \Ext_A^i(A/I^k,D) \rightarrow \Ext_A^i(A/I^k,L) \rightarrow \Ext_A^i(A/I^k,L_0) \rightarrow \cdots.$$
Since $D$ is $C$-divisible and $C$-torsion-free, $\Ext^i_A(A/I^k,D) =0$ by the above claim, and this shows that the mapping $\Ext_A^i(A/I^k,L) \rightarrow \Ext_A^i(A/I^k,L_0)$ is an isomorphism. Therefore, $\Ext_A^i(A/I^k,L) \rightarrow \Ext_A^i(A/I^k,L')$, as a composition of two isomorphisms, is an isomorphism.

Finally, since the mapping $A \rightarrow A'$ is $C$-analytic, we have 
 $A/I^k \cong (A/c^kA)/(I^k/c^kA) \cong (A'/c^kA')/(I^kA'/c^kA') \cong A'/I^kA'.$
Thus since $A'$ is a flat $A$-module \cite[Theorem 8.8, p.~160]{Ma}, we have by a change of rings \cite[Chapter VI, Proposition 4.1.3]{CE}:
 $$\Ext_{A'}^i(A'/I^kA',L') \cong \Ext_A^i(A/I^k,L').$$
Since these isomorphisms are natural,  we conclude that $$H^i_I(A) \cong \lim_{\rightarrow} \Ext_A^i(A/I^k,L) \cong \lim_{\rightarrow} \Ext_{A'}^i(A'/I^kA',L') \cong H^i_{IA'}(A').$$
This proves the lemma.
\end{proof}

\index{depth} \index{$\depth_I(L)$}

Using the lemma, we show now that the local cohomology of $R$ can be expressed in terms of
that of $S$ and $K$. Note that in the theorem we do not need that $I$ is a contracted ideal; this is a consequence of the fact that local cohomology is determined up to radical.

\begin{thm} \label{local cohomology} Let $R$ and $S$ be  local Noetherian rings with  maximal ideals $M$ and $N$, respectively, and let $C$ be a multiplicatively closed subset of nonzerodivisors of $S$ meeting the maximal ideal of $S$.    Suppose that $R$ is twisted along $C$ by a $C$-torsion-free module $K$.   If $I$ is an ideal of $R$ meeting $C$, then for each $i\geq 0$, there is an isomorphism of $R$-modules, $$H_I^i(R) \cong H_{IS}^i(S) \oplus H_{IS}^i(K).$$    \end{thm}

\begin{proof} Since local cohomology is the same up to radical, we may replace $I$ with its radical in $R$ \cite[Remark 1.2.3, p.~5]{BrSh}. Since $R \subseteq S$ is integral, each radical ideal of $R$ is contracted from $S$, and hence we may assume that $I$ is an ideal of $R$ meeting $C$ that is contracted from an ideal of $S$.
Let $c \in I \cap C$, and let $R'$, $S'$ and $K'$ be the $(c)$-adic completions of $R$, $S$ and $K$, respectively. 
By Lemma~\ref{Groth}, $H_I^i(R) \cong H_{IR'}^i(R')$.  Also,
by (2.6), there is a ring isomorphism $f:R' \rightarrow S' \star K'$ such that since $I$ is contracted
 from $S$, we have by Lemma~\ref{product lemma},  $f(I)(S'\star K') = IS' \star K'$.
 Thus 
  $$H_{I}^i(R) \cong H_{(IS') \star K'}(S' \star K').$$ Again since  local cohomology is the same up to radical, then $$H_{I}^i(R) \cong H_{(IS') \star K'}(S' \star K') \cong H_{(IS') \star (IK')}(S'\star K').$$
 Now let $\iota:S' \rightarrow S' \star K'$ denote the  injection $s \mapsto (s,0).$  Then under the $S'$-module structure induced on $S' \star K'$ by $\iota$, we have $$IS'(S' \star K') = (IS' \star 0)(S'\star K') = IS' \star IK'.$$
 Therefore, 
  by Independence of Base \cite[4.2.1, p.~71]{BrSh}, $$H_{I}^i(R) \cong H_{(IS') \star (IK')}^i(S' \star K') \cong H_{IS'}^i(S' \oplus K').$$ Since local cohomology commutes with direct sums \cite[Proposition 7.3, p.~78]{Iye}, we conclude  $$H_I^i(R) \cong H_{IS'}^i(S' \oplus K') \cong H_{IS'}^i(S') \oplus H_{IS'}^i(K').$$ An application of Lemma~\ref{Groth} yields $H^i_{IS'}(S') \cong H^i_{IS}(S)$ and $H^i_{IS'}(K') \cong H^i_{IS}(K)$, and this proves the theorem.
 \end{proof}

If $I$ is an ideal of the  local Noetherian ring $A$ and  $L$ is an $A$-module, then $\depth_I(L)$ is the greatest integer $i$ such that for all $j < i$, $H^j_I(L)=0$.  When $L$ is a finitely generated $A$-module, then $\depth_I(L)$ is the length of a maximal regular sequence on $L$ \cite[Theorem~6.27]{BrSh}.  Regardless of whether $K$ is a finitely generated $S$-module, we have from the theorem: \index{depth} \index{$\depth_I(L)$}

\begin{cor} \label{depth} With the same assumptions as the theorem,  $$\depth_I(R) = \min\{\depth_{IS}(S),\depth_{IS}(K)\}. $$ 
\end{cor}

\section{Subrings twisted by a finitely generated module}

In \cite[Theorem 6.1]{OlbCounter}, it is shown that when $S$ is a local Noetherian domain and $R$ is a subring of $S$ strongly twisted by a  finitely generated torsion-free $S$-module $K$, then $R$ is Cohen-Macaulay if and only if $S$ is Cohen-Macaulay and $K$ is a maximal Cohen-Macaulay module; $R$ is  Gorenstein  if and only  if  $S$ is  Cohen-Macaulay and admits a canonical module $\omega_{S}$ with $K \cong \omega_S$;  
  $R$ is a complete intersection if and only if $S$ is a complete intersection and $K \cong S$; and 
 $R$ is a hypersurface if and only if $S$ is a regular local ring and $K \cong S$.

Motivated by these descriptions, 
 we consider subrings of a local Noetherian domain twisted by a finitely generated module. 
Although the case of strongly twisted subrings motivates this section, we work under the following more general  assumption.

\begin{quote} {\it $S$ is a local Noetherian domain with maximal ideal
 $N$ and quotient field $F$;  $K$ is a nonzero {\bf finitely generated}
 torsion-free $S$-module; and $R$ is a Noetherian subring of $S$ twisted by $K$ along some multiplicatively closed subset $C$ of nonzerodivisors of $S$  that meets $N$.}
\end{quote}

Note that since $N$ meets $C$ and $R \subseteq S$ is integral, then $R$ is quasilocal with maximal ideal $M$ meeting $C$. 
We first calculate embedding dimension and multiplicity of $R$ in terms of $S$ and $K$.  
 Let $(A,{\ff m})$ be a local Noetherian ring, let $J$ be an ${\ff m}$-primary ideal of $A$, and let $L$ be a finitely generated $A$-module. 
Denote by  \index{Hilbert polynomial}  $P_{J,L}(X) \in {\mathbb{Q}}[X]$  the {Hilbert polynomial}  of $J$ on $L$; that is, $P_{J,L}(X)$ has the property that for $n \gg 0$, $P_{J,L}(n) = H_{J,L}(n)$ \cite[Proposition 12.2 and Exercise 12.6]{Eis}.  
We write $e(J,L)$ for the multiplicity of $J$ on $L$.
The \index{multiplicity of a local ring} {multiplicity of the local ring} $(A,{\ff m})$ is denoted $e(A)$, and is defined by $e(A)= e({\ff m},A)$.

\begin{thm} \label{fg invariants} Let $I$ and $J$ be ideals of $R$ contracted from $S$ such that $J$ is $M$-primary and $I$ meets $C$.  
\begin{itemize}

\item[{(1)}]   $\mu_R(I) = \mu_S(IS) + \mu_S(K)$.  

\item[{(2)}]   $\embdim R = \embdim S + \mu_S(K).$

\item[{(3)}]   $e(I,J) = e(IS,S) \cdot (1+{\rm rank}(K)).$

\item[{(4)}]    $e(R) = e(S) \cdot (1+{\rm rank}(K))$.

\end{itemize}
\end{thm}

\begin{proof} 
(1)  Since $K$ is finitely generated, $\mu_S(K) = \dim_{S/N}K/NK$, so (1) follows from Theorem~\ref{pre-construction Hilbert}(2). 

(2) By (2.4), $N = MS$, where $M$ is the maximal ideal of $R$, so (2) is clear.  

(3)  As discussed before the theorem, there exists a unique polynomial $P_{J,I}$ such that $P_{J,I}(i) = H_{J,I}(i)$ for all $i\gg 0$.  Similarly, since by assumption, $K$, and hence $(I+J)K$, are finitely generated, there exists a unique polynomial $P_{JS,(I+J)K}$ such that $$P_{JS,(I+J)K}(i) = H_{JS,(I+J)K}(i) \: {\mbox{ for all }} i\gg 0.$$  Define a polynomial $P^*_{JS,(I+J)K}(X) \in {\mathbb{Q}}[X]$ by $$P^*_{JS,(I+J)K}(X) = P_{JS,(I+J)K}(X-1).$$
 Then by Theorem~\ref{pre-construction Hilbert} we have that for sufficiently large $i$,
 \begin{eqnarray*}H_{J,I}(i)  & = & H_{JS,IS}(i) + H_{JS,(I+J)K}(i-1) \\
 \: & = & P_{JS,IS}(i) + P_{JS,(I+J)K}(i-1) \\
 \: & = & P_{JS,IS}(i) + P^*_{JS,(I+J)K}(i).
 \end{eqnarray*}
Thus $P_{JS,IS} + P^*_{JS,(I+J)K}$ is the Hilbert polynomial $P_{J,I}$ for $H_{J,I}$.    Let $d$ denote the Krull dimension of $S$.  Note that the degree and leading coefficient of $P^*_{JS,(I+J)K}$ are the same as that of $P_{JS,(I+J)K}$.
 Since $(I+J)K$ is a torsion-free $S$-module, $(I+J)K$ has dimension $d$ as an $S$-module, so that $P^*_{JS,(I+J)K}$  has degree $d-1$.  Moreover, the leading coefficient of $P^*_{JS,(I+J)K}$  is $e(JS,(I+J)K)/(d-1)!$.   Thus the leading coefficient of $P_{J,I}$ is $$\frac{e(JS,IS) + e(JS,(I+J)K)}{(d-1)!},$$ and hence we arrive at $$e(J,I) = e(JS,IS) + e(JS,(I+J)K).$$  But since $K$ is a finitely generated torsion-free module over the domain $S$, and $JS$ is an $N$-primary ideal of $S$, we can expand this further by using the fact that $e(JS,IS) = e(JS,S)\cdot \rank(IS) = e(JS,S)$ \cite[Theorem 14.8, p.~109]{Ma}.
Moreover, $$e(JS,(I+J)K) = e(JS,S) \cdot \rank((I+J)K) = e(JS,S) \cdot \rank(K),$$ where the last equality follows from the fact that since $K$ is torsion-free,  it has the same rank as $(I+J)K$.
 Therefore, \begin{eqnarray*}
 e(J,I) & = &  e(JS,IS) + e(JS,(I+J)K) \\
 \: & = & e(JS,S) + e(JS,S) \cdot \rank(K) \\
 \: & = & e(JS,S)\cdot (1+ \rank(K)),\end{eqnarray*} and this completes the proof of (3).
 
 (4) This  is clear in light of  (3) and the fact from (2.4) that $N = MS$, where $M$ is the maximal ideal of $R$. 
\end{proof}

As noted at the beginning of the section, when $S$ is Cohen-Macaulay and $R$ is strongly twisted by a maximal Cohen-Macaulay module $K$, then $R$ is Cohen-Macaulay. 
Even in this case, the obstacle to saying more  about the possibilities for the embedding dimension and multiplicity of $R$ is knowledge about the sizes of the maximal Cohen-Macaulay $S$-modules.    There are some limitations imposed on these modules.  If $(A,{\ff m})$ is a local Cohen-Macaulay domain and $L$ is a maximal Cohen-Macaulay $A$-module, then $\mu_A(L) \leq e({\ff m},L)$  \cite[Proposition 1.1]{BHU}. If this bound is attained, that is, if $\mu_A(L) = e({\ff m},L)$, then $L$ is said to a \index{maximal Cohen-Macaulay module!maximally generated} {\it maximally generated maximal Cohen-Macaulay module}.

For example, if $S$ is a regular local ring, then every maximal Cohen-Macaulay module is free and $e(S) = 1$, so every maximal Cohen-Macaulay module is necessarily maximally generated.  However, given a local Cohen-Macaulay ring, it is difficult in general to determine whether there exist maximally generated maximal Cohen-Macaulay modules; see \cite{BHU}.
  But granted their existence for our specific ring $S$, we can  say a little more about the embedding dimensions and multiplicities of $R$ and $S$:

\begin{cor} \label{MGMCM} If $S$ is a Cohen-Macaulay ring that admits a maximally generated maximal Cohen-Macaulay module, and $K$ is chosen to be this module, then  $R$ is a
  Cohen-Macaulay ring with $\embdim R - \embdim S = e(R) - e(S).$
\end{cor}

\begin{proof} By Theorem~\ref{fg invariants} and the fact that $\mu_S(K) = e(N,K)$, we have $$\embdim R - \embdim S = \mu_{S}(K) = e(N,K) = e(S) \cdot \rank(K) = e(R) - e(S).$$    
\end{proof}

Conversely, if $R$ is  Cohen-Macaulay and $\embdim R - \embdim S = e(R) - e(S)$, then from Theorem~\ref{fg invariants},
 we deduce that $\mu_S(K) = e(S) \cdot \rank(K) = e(N,K)$.
 %
%
%
%
  Thus in our special setting the existence of such maximally generated maximal Cohen-Macaulay modules is equivalent to the existence of a twisted Cohen-Macaulay subring of $S$ whose embedding dimension differs from that of $S$ by the same amount as its multiplicity differs from that of $S$.

If $A$ is a local Cohen-Macaulay ring of Krull dimension $d$, then an inequality due to Abhyankar \index{Abhyankar, S.} \cite{AbhLocal} places a lower bound on the multiplicity of $A$: $$e(A) \geq \embdim A - d+1.$$  When the lower bound is attained, that is, when equality holds, then $A$ has {\it minimal multiplicity}. 
 In our context where $R$ is a twisted subring of $S$ of Krull dimension $d$, we have from Theorem~\ref{fg invariants} that $$ \embdim R - d + 1  =  \embdim S + \mu_S(K) - d + 1.$$ Thus since by Theorem~\ref{fg invariants}, $e(R) = e(S) \cdot (1+\rank(K))$, we see that $R$ is a Cohen-Macaulay ring of minimal multiplicity if and only if $S$ is a Cohen-Macaulay ring and $K$ is a maximal Cohen-Macaulay module  with $$e(S) + e(S) \cdot \rank(K)\: = \: \embdim S + \mu_S(K) - d + 1,$$ or equivalently,
$$e(S) - (\embdim S - d +1) \: = \:   \mu_S(K) - e(S) \cdot \rank(K).$$  By Abhyankar's inequality, the left hand side is never less than $0$, and similarly, as discussed above, $\mu_S(K) \leq e(S) \cdot \rank(K)$, so that the right hand side is  never more than $0$.  Therefore, $R$ has minimal multiplicity if and only if $e(S) = \embdim S - d +1$ and $\mu_S(K) = e(S) \cdot \rank(K).$ 
 Putting all this together, we have:

\begin{cor} \label{min mult} $R$   is  Cohen-Macaulay ring of minimal multiplicity if and only if  $S$ is a Cohen-Macaulay ring of  minimal multiplicity and  $K$ is a maximally generated maximal Cohen-Macaulay module. \qed
\end{cor}


When $A$ is a Cohen-Macaulay ring with infinite residue field, then $A$ has minimal multiplicity if and only if the maximal ideal of $A$ has reduction number $\leq 1$ \cite[Exercise 4.6.14, p.~192]{BH}.
Using this fact, 
 we  give the example promised after Corollary~\ref{minimal red} of a situation in which $r_R(I) = r_S(IS)$.  The idea is to choose $(S,N)$ to be a local ring with infinite residue field and  minimal multiplicity, but such  that $S$ is not a regular local ring.  Then as discussed above, the fact that $S$ has minimal multiplicity implies that $N$ has reduction number $\leq 1$. In fact, since $S$ is not a regular local ring  it must be that $r_S(N) = 1$.  If we choose $K$ to be a maximally generated maximal Cohen-Macaulay $S$-module, then by Corollary~\ref{min mult}, $R$ also has minimal multiplicity and is not regular, so that $r_R(M) = 1$.  Since by (2.4), $N = MS$, we have then that $r_R(M) = 1 = r_S(MS)$.

Thus to illustrate Corollary~\ref{minimal red}, all that remains to do is to show that $S$ can be chosen in conformance with (1.1) in such a way that $S$ has minimal multiplicity, $S$ has infinite residue field, $S$ is not regular, and $S$ possesses a maximally generated maximal Cohen-Macaulay module.

\begin{exmp} \label{second reduction example} \index{strongly twisted subring!existence of}
\index{strongly twisted subring!characteristic $p$}
\index{strongly twisted subring!of minimal multiplicity}
{\em Let $k$ be a field of characteristic $p \ne 0$ that is separably generated of infinite transcendence degree over a countable subfield, and let $T,X_1,\ldots,X_n$ be indeterminates for $k$.  Define $$S = k[T^2,T^3,X_1,\ldots,X_n]_{(T^2,T^3,X_1,\ldots,X_n)}.$$  Then $S$ has quotient field $k(T,X_1,\ldots,X_n)$, so that for any choice of a finitely generated torsion-free $S$-module $K$, there exists by (1.1) a subring of $S$ strongly twisted by $K$.  Since $S$ is the localization of a polynomial ring over the Cohen-Macaulay ring $k[T^2,T^3]$, $S$ is also a Cohen-Macaulay ring.  Thus since $S$ has infinite residue field, to show that $S$ has minimal multiplicity it suffices to show that the maximal ideal $N$ of $S$ has reduction number $1$.
  Since $(T^2,T^3)^2 = T^2(T^2,T^3)$ in $S$ (or even in $k[T^2,T^3])$, it follows that $N^2 = (T^2,X_1,\ldots,X_n)N$, and hence $N$ has reduction number at most $1$.  Therefore, $S$ is a Cohen-Macaulay ring of minimal multiplicity that is not a regular local ring (it is not even integrally closed).
All that remains now is to exhibit a maximally generated Cohen-Macaulay $S$-module $K$. Since $S$ has infinite residue field, this amounts by Lemma 1.3 in \cite{BHU} to finding  a finitely generated torsion-free $S$-module $K$ such that $IK = NK$ for some ideal $I$ of $S$ generated by a regular sequence.  Choosing $K = N$, we have then since $N$ has reduction number $1$ that $K$ is a  maximally generated maximal Cohen-Macaulay module, and hence $S$ satisfies all the requirements discussed before the example.  Therefore, $r_R(M) = r_S(MS)$, and this shows that the first case in Corollary~\ref{minimal red} occurs. \qed}

\end{exmp}



\section{Subrings twisted by a valuation ring}

\label{(V)} \label{V section}

In this section we no longer assume $K$ is a finitely generated
torsion-free module.  Instead, we work with an $S$-module $K$ that
although not finitely generated, has the property that $K/sK$ is
finitely generated for all $0 \ne s \in S$.  This then by
(2.9)  guarantees that when $S$
is a Noetherian ring, then a subring $R$ of $S$ strongly twisted
by $K$ is also Noetherian.  
Recall that  a {\it DVR}  is a rank one discrete valuation ring; \index{DVR}  equivalently, a DVR  is a local PID.   The DVR $V$ {\it birationally dominates} \index{birationally dominates} the local ring $S$ if $V$ is an overring of $S$ (and hence has the same quotient field as $S$) such that the maximal ideal of $V$ contracts to the maximal ideal of $S$; or, equivalently, $NV \ne V$, where $N$ is the maximal ideal of $S$.  Thus if $V$ birationally dominates $S$, we may consider the nonzero $S/N$-vector space $V/NV$.  We are interested in the case where  $V$ is {\it residually finite}; that is, when 
the vector space $V/NV$ has  finite dimension.  For then since $V$ is a DVR, it follows that for every $0 \ne s \in S$, $V/sV$ is a finitely generated $S$-module.    In fact,  when $K$ is a torsion-free finite rank $V$-module (with $V$ residually finite), then $K/sK$ is a finitely generated $S$-module for all $0 \ne s \in S$ \cite[Lemma 5.4]{OlbCounter}.

The existence of 
 such a DVR  depends on  the generic formal fiber of the local domain $S$.  Viewing the quotient field $F$ of $S$ as a subring of  the total quotient ring of $\widehat{S}$, the ring $\widehat{S}[F]$ is the {\it generic formal fiber} of  $S$. 
If $S$ has Krull dimension $d>0$, then the generic formal fiber has Krull dimension at most $d-1$.
Heinzer, Rotthaus and Sally  have shown in \cite[Corollary 2.4]{HRS} that    a birationaly dominating residually finite DVR exists if and only if the dimension of the generic formal fiber of $S$ is one less than the dimension of $S$.  Matsumura proved in \cite[Theorem 2]{Mat2} that the latter requirement is satisfied whenever $S$ is the localization at a maximal ideal of an affine $k$-domain,  where $k$ is a field and  the affine domain has Krull dimension $>1$.

In this section we assume the presence of such a residually finite DVR that birationally dominates $S$. Specifically, our assumptions for this section are: 
 
\begin{quote}{\it  $S$ is a local Noetherian domain with maximal ideal $N$ and quotient field $F$; $V$ is a DVR  such that
 $V$   birationally dominates $S$ and  $V/NV$ is a finite extension of $S/N$ of degree $m$; 
 and   $R$ is a subring of $S$ that is twisted by a nonzero torsion-free finite rank $V$-module $K$ along some multiplicatively closed subset $C$ of $S$ containing a nonunit of $S$.
  } \end{quote}

    These assumptions are the same as the last section with the  exception that rather than assume     $K$ is a finitely generated $S$-module, we assume  $K$ is a torsion-free
      finite rank $V$-module. We first prove a version of Theorem~\ref{fg invariants} for the present case. 
          Let ${\ff M}_V$ denote the maximal ideal of $V$, and let  $$r_K = \dim_{V/{\ff M}_V} K/{\ff M}_VK.$$ 
          By \cite[Lemma 5.4]{OlbCounter}, 
     for each $0 \ne x \in {\ff M}_V$, $K/xK$ is a free $V/xV$-module of rank $r_K$.  We use this invariant of $K$ in calculating the embedding dimension and multiplicity of $R$.

\begin{thm} \label{V invariants}
Let  $I$ and $J$ be  ideals of $R$ contracted from $S$ such that  $J$ is $M$-primary and $I$ meets $C$. 

\begin{itemize}

\item[{(1)}]   $\mu_R(I) = \mu_S(IS) + m \cdot r_K.$    

\item[{(2)}]  $\embdim R = \embdim S + m \cdot r_K.$

\item[{(3)}]  The multiplicity of $J$ on $I$ is  given by$$
e(J,I) = e(JS,S) +  \left\{ \begin{array}{ll}
r_K \cdot \length V/JV \; &  \mathrm{if}\, \dim(S) =1  \\
 0\; & \mathrm{if}\, \dim(S) >1.
\end{array}\right.
$$

\item[{(4)}]  The multiplicity of the local ring $R$ is   $$
e(R) = e(S) +  \left\{ \begin{array}{ll}
m\cdot r_K \; &  \mathrm{if}\, \dim(S)=1  \\
 0\; & \mathrm{if}\, \dim(S)>1. 
\end{array}\right.
$$

\end{itemize}
\end{thm}  

\begin{proof} (1) 
  By
Theorem~\ref{pre-construction Hilbert},  the minimal number of generators of $I$ is 
\begin{eqnarray*} 
\mu_R(I) & = & \mu_S(IS) + \dim_{S/N} K/NK  \\
\: &  = &  \mu_S(IS) + \dim_{S/N}V/NV \cdot \dim_{V/NV}K/NK \\
\: & = & \mu_S(IS) + m \cdot r_k.
\end{eqnarray*}

(2) Since by (2.4), $N =MS$, statement (2) follows from (1).

(3)  By Theorem~\ref{pre-construction Hilbert}, we have for each $n>0$,
$$H_{J,I}(n) = H_{JS,IS}(n) + H_{JS,(I+J)K}(n-1).$$
As discussed before Theorem~\ref{pre-construction Hilbert}, there exists a   unique polynomial $P_{JS,IS}$ with rational coefficients such that $P_{JS,IS}(i) = H_{JS,IS}(i)$ for all $i\gg 0$.  Moreover, $P_{JS,IS}$ has degree $d-1$, where $d$ is the Krull dimension of $S$, and  the leading coefficient of $P_{JS,IS}$ is $e(JS,IS)/(d-1)!$.  Thus for  sufficiently large $i$, we have
\begin{eqnarray*}
H_{J,I}(i)  & = &   P_{JS,IS}(i) + H_{JS,(I+J)K}(i-1) \\
\: & = &  \frac{e(JS,IS)}{(d-1)!} i^{d-1} + {\mbox{ terms in }}P_{JS,IS}{\mbox{ of lower degree }} \\ \: & \: & \:\:\:\:\:\:\:\:\:\:\:\:\:\:\:\:\:\:\:\:\:\:\:\:\:\:\:\:\: + \: \: H_{JS,(I+J)K}(i-1).
\end{eqnarray*}
It remains to examine the Hilbert function: $$H_{JS,(I+J)K}(i) = \length  J^{i}(I+J)K/ J^{i+1}(I+J)K, {\mbox{ where }} i>0.$$   Since, as discussed before the theorem, $K/JK$ is a free $V/JV$-module of rank $r_K$,  and every ideal of $V$ is a principal ideal, we have:
 \begin{eqnarray*}
 H_{JS,(I+J)K}(i) & = & \length  J^{i}(I+J)K/ J^{i+1}(I+J)K \\
 \:& = & \length K/JK \\
 \: & = & \length \bigoplus_{j=1}^{r_K} V/JV \\
 \: & = & r_K \cdot \length V/JV.
 \end{eqnarray*}
In particular, $H_{JS,(I+J)K}(i)$ is constant for all $i>0$.
  This means that for sufficiently large $i$, we have:  \begin{eqnarray*}
  P_{J,I}(i) & = &
  H_{J,I}(i) \\
  & = &  \frac{e(JS,IS)}{(d-1)!} i^{d-1}+ {\mbox{ terms in }} P_{JS,IS} {\mbox{ of lower degree }}\\ \: & \: & \:\:\:\:\:\:\:\:\:\:\:\:\:\:\:\:\:\:\:\:\:\:\:\:\:\:\:\:\: + \: \:  r_K \cdot \length V/JV.
  \end{eqnarray*}
Therefore,
if the Krull dimension $d$ of $S$ is $1$ so that $d-1 = 0$, we have that the leading coefficient (and constant term) of $P_{J,I}$ must be given by
$$e(J,I) = e(JS,IS) + r_K \cdot \length V/JV.$$
On the other hand, if $d>1$, then  the constant term $r_K\cdot \length V/JV$ contributes nothing to the leading coefficient of $P_{J,I}$, so we have
$e(J,I) = e(JS,IS).$ As discussed in the proof of Theorem~\ref{pre-construction Hilbert}(3), $E(JS,IS) = E(JS,S)$, so  we have justified  the stated multiplicity formula.

(4) This is clear from (3) and the fact that $N = MS$.
\end{proof}

Under the additional assumption that $V = S + NV$ (so that $m = 1$), then every $S$-module between $K$ and $FK$ is also a $V$-module (see  the proof of \cite[Proposition 5.6]{OlbCounter}), and hence 
 every ring between $R$ and $S$ is a local Noetherian ring twisted along $C$ by a $V$-module  between $K$ and $K_C$ \cite[Proposition 5.6]{OlbCounter}. It follows from 
 Theorem~\ref{V invariants} that the  multiplicity of the rings between $R$ and $S$ decreases from $e(S)+r_k$ to  $e(S)$, where the last multiplicity is obtained only for $S$.

The special case $K = V$ is also of interest, especially when  $V= S + NV$. In this case, as above, all the $S$-modules properly between $K= V$ and $FK = F$ are of the form $v^{-1}V$ for some $0 \ne v \in V$.  In particular, these modules form a well-ordered chain. 
 It follows that the subrings $T$ of $S$ twisted by these modules also lie in a corresponding well-ordered chain.  This, along with the   facts that every $R$-submodule of $S$ containing $R$ is a ring (2.3) and every ring between $R$ and $S$ is twisted by a module between $K$ and $F$ \cite[Theorem 4.3]{OlbCounter}, imply
  that $T$ can be written as $T= R + tR$ for some $t \in T$.

 In summary:
 Suppose that $K = V$, $V = S + NV$ and $S$ has Krull dimension $>1$.  Then the rings $T$ between $R$ and $S$ are totally ordered by inclusion, and can be written $$R = T_0 \subsetneq T_1 \subsetneq T_2 \subsetneq \cdots \subsetneq S,$$ where for each $i$, $T_i$ is a local Noetherian subring of $S$ twisted along $C$.   Moreover, for each $i$, there exists $t_i \in T_i$ such that $T_i= R + t_iR$, and   $$e(T_i) = e(S) \: {\mbox{ and }} \: \embdim(T_i) = \embdim(S) + 1.$$  
We could in fact descend from $R$ also using the ideals $vV$, $0 \ne v \in V$.

Returning to the general case of the standing assumption for this section, the fact that $K$ is a module over a DVR makes it easy to calculate the local cohomology of $R$ in terms of that of $S$:

\begin{thm} \label{V cohomology}  The local cohomology modules for an ideal $I$ of $R$ meeting $C$ are given by  $$
H_I^i(R) \cong  \left\{ \begin{array}{ll}
0 \; &  \mathrm{if}\, i=0 \\
 H^1_{IS}(S) \oplus FK/K \; & \mathrm{if}\, i=1 \\
 H^i_{IS}(S) \; & \mathrm{if}\, i>1
\end{array}\right.
$$  
\end{thm}

\begin{proof}  By Theorem~\ref{local cohomology}, all that needs to be shown is that $H_{IS}^0(S) = H_{IS}^0(K)  = 0$, $H^1_{IS}(K) = FK/K$ and $H^i_{IS}(K)=0$ for all $i>1$. The first case, where $i=0$, is clear, since $S$ and $K$ are torsion-free $S$-modules.   To deal with the other cases, we can  by Independence of Base (see the proof of Theorem~\ref{local cohomology})  pass to the ring $V$ via an isomorphism: $H^i_N(K) \cong H^i_{NV}(K)$.  Since $V$ has Krull dimension $1$, $H^i_{NV}(K) =0$ for all $i>0$ (Grothendieck's Vanishing Theorem \cite[Theorem 6.1.2, p.~103]{BrSh}).
Thus since  $$H^1_{NV}(K)\cong \lim_{\rightarrow}\Ext_V^1(V/N^kV,K),$$
it is enough to show that $$\lim_{\rightarrow} \Ext_V^1(V/N^kV,K) \cong FK/K.$$  Since $V$ is a DVR, there exists $x \in N$ such that $NV=xV$.
 Consideration of the exact sequence, $$0 \rightarrow \Hom_V(V,K) \rightarrow \Hom_V(x^kV,K) \rightarrow \Ext_V^1(V/x^kV,K) \rightarrow 0,$$  yields a natural isomorphism $x^{-k}K/K \cong  \Ext_V^1(V/x^kV,K)$.  Thus since $F = V[x^{-1}]$, we obtain $$H^1_{NV}(K)\cong \lim_{\rightarrow}x^{-k}K/K \cong FK/K,$$ which verifies the theorem.
 \end{proof}

The calculation in Theorem~\ref{V cohomology} also illustrates a striking way in which $R$ fails to be Cohen-Macaulay whenever $R$ has Krull dimension $>1$:

\begin{cor}   With $M$ the maximal ideal of $R$, if $R \subsetneq S$, then  $\depth_M(R) = 1$.
\end{cor}

\begin{proof}  This follows from the depth calculation
in Corollary~\ref{depth}: $$\depth_M(R) = \min\{\depth_N(S),\depth_N(K)\}.$$ Indeed,
based on this expression of $\depth_M(R)$, it is enough to show that
$H^1_M(R) \ne 0$, and this is clear from the theorem, since we have
assumed $R \subsetneq S$ and hence that $K \subsetneq FK$.
\end{proof}

We illustrate some of these ideas with the following example, which is based on an existence result for Noetherian twisted subrings from \cite{OlbCounter}, that in turn is based on an argument of Ferrand and Raynaud \cite{FR}. It differs from most of our other examples in that in order for $R$ to be Noetherian, we do not require  $R$ to be a strongly twisted subring of $S$, only that $R$ is twisted along a ``small'' multiplicatively closed set. The idea is as follows. 
 Suppose $(S,N)$ is a two-dimensional local Noetherian UFD  that  is birationally   dominated by a DVR $V$ having the same residue field as $S$, and such that there is $t \in N$ with  $tV$ the maximal ideal of $V$.  If  $R$ is a subring of $S$ that is twisted along $C =\{t^i:i>0\}$ by    an $S$-submodule $K$ of a  finitely generated free $\widehat{V}$-module $K'$ with $K'/K$ a $C$-torsion-free $S$-module, then $R$ is a two-dimensional local Noetherian ring such that for every height $1$ prime ideal $P$ of $R$, $R_P = S_Q$ for some height $1$ prime ideal $Q$ of $S$ \cite[Theorem 5.7]{OlbCounter}.

\begin{exmp} \label{UFD example}  {\em \index{twisted subring!of a regular local ring} \index{twisted subring!of characteristic $0$}
Let $k$ be a field of characteristic $0$ having infinite transcendence degree over its prime subfield, and let $S = k[X,Y]_{(X,Y)}$. There exists a DVR overring $V$ of $S$ with residue field  $k$ and whose maximal ideal is generated by the maximal ideal $N = (X,Y)S$ of $S$; see for example \cite[p.~102]{ZS}.   Without loss of generality, assume $NV = XV$.  As discussed before the example, for every finite rank free $V$-module $K$, there is a local Noetherian \index{twisted subring!Noetherian} subring $R$ of $S$ that is twisted by $K$ along $C = \{X^i:i\geq 0\}$. \index{twisted subring!embedding dimension of} Thus using the results of this section we have for each $n>0$ there is a local Noetherian subring $R$ of $S$ such that \index{twisted subring!and multiplicity}
  $\embdim R = 2+n$ and $e(R) = 1$.   Moreover,  $R \subseteq S$ is a quadratic extension; $R$ has quotient field $F$ and  $\widehat{R}$ has nilpotents  and an embedded prime.
 Finally,   every ring between $R$ and $S$ is Noetherian. \qed
 }
\end{exmp}

In the example, $(S,N)$ could \index{twisted subring!of a UFD} be replaced by any UFD of characteristic $0$ and Krull dimension $2$ whose quotient field has infinite transcendence degree over its prime subfield, as long as there exists a DVR $V$ that birationally dominates $S$ and has $V= S + NV$.  If instead of this last condition we assume that $[V/NV:S/N] = m < \infty$, then the properties of $R$ in the example all remain true, with the exception that the embedding dimension would now be $2 + mn$, and our previous argument for $V = S +NV$ no longer  guarantees that the rings between $R$ and $S$ are Noetherian.
With different assumptions on the field  $k$, we can say more:

\begin{exmp} \label{UFD example 2}\index{separably generated extension} {\em Let $k$ be a field of characteristic $p$, and suppose  $k$ is  separably generated \index{twisted subring!of characteristic $p$} and of infinite transcendence degree over a countable subfield. \index{strongly twisted subring!of a regular local ring}  Let $S = k[X_1,\ldots,X_d]_{(X_1,\ldots,X_d)}$, where $d>1$.  Then there exists a DVR $V$  \index{DVR!birationally dominating} as in Example~\ref{UFD example}, and for each torsion-free $V$-module $K$ of finite rank $n>0$, \index{strongly twisted subring!Noetherian}
there exists by (1.1) and (2.9) a  local Noetherian subring $R$ of $S$ strongly twisted by $K$ and satisyfing all \index{strongly twisted subring!prime ideal of}
  the assertions in Example~\ref{UFD example}, \index{strongly twisted subring!embedding dimension of} with the exception \index{strongly twisted subring!integrally closed ideal of} that $\embdim(R) = d + n$. \index{strongly twisted subring!and regularity condition $R_i$} Also, $R \subseteq S$ is subintegral and as discussed above every ring between $R$ and $S$ is a local Noetherian ring strongly twisted by an $S$-module between $K$ and $FK$.
 \qed
   }
\end{exmp}





\medskip

{\it Acknowledgment.} I thank the referee for helpful comments that improved the presentation of the paper.

\end{document}